\documentclass[12pt]{article}

\usepackage{amssymb}
\usepackage{latexsym}
\usepackage{amsmath}
\usepackage{indentfirst}

\title{3-colored asymmetric bipartite Ramsey number of connected matchings and cycles}
\author{Zhidan Luo \thanks{ College of Mathematics, Hunan University, Changsha 410082, P.R. China. Email: luodan596@hnu.edu.cn}
\and Yuejian Peng \thanks{Corresponding author. College of Mathematics, Hunan University, Changsha 410082, P.R. China. Email: ypeng1@hnu.edu.cn.
Partially  supported by National Natural Science Foundation of China (No. 11671124).}
}
\date{}

\newtheorem{defi}{Definition}[section]
\newtheorem{theo}{Theorem}[section]

\newtheorem{remark}[theo]{Remark}
\newtheorem{lemma}[theo]{Lemma}

\newtheorem{fact}[theo]{Fact}

\def\q{\hspace*{\fill}$\Box$\medskip}

\begin{document}
\maketitle
\begin{abstract}
Let $k,l,m$ be integers and $r(k,l,m)$ be the minimum integer $N$ such that for any red-blue-green coloring of $K_{N,N}$, there is a red  matching of size at least $k$ in a component, or a blue matching of at least size $l$ in a component, or a green matching of size at least $m$ in a component.
In this paper, we determine the exact value of $r(k,l,m)$ completely. Applying a technique originated by {\L}uczak that applies Szemer\'edi's Regularity Lemma to reduce the problem of showing the existence of a monochromatic cycle to show the existence of a monochromatic  matching in a component, we obtain the 3-colored asymmetric bipartite Ramsey number of cycles asymptotically.
\end{abstract}

\section{Introduction}

Let $k\geq 2$ be an integer and $H_{1},...,H_{k}$ be graphs. The Ramsey number $R(H_{1},...,H_{k})$ is the minimum integer $N$ such that any $k$-edge-coloring of $K_{N}$ contains a monochromatic $H_{i}$ in color $i$ for some $1\leq i\leq k$. If $H_{1}=H_{2}=...=H_{k}=H$, we simplify the notation as $R_{k}(H)$.

In 1967, Gerencs\'er and Gy\'arf\'as \cite{GG} showed that $R_{2}(P_{n}) =\lfloor \frac{3n}{2}-1 \rfloor$ where $P_{n}$ is a path on $n$ vertices. Bondy and Erd\H{o}s~\cite{JB73},  Faudree and Schelp \cite{FS}, and Rosta \cite{R} determined the 2-colour Ramsey number of cycles. The case of 3-coloring is more difficult and there is almost no result until 1999, {\L}uczak \cite{L} determined that for odd cycles and showed that $R(C_{n},C_{n},C_{n})\leq (4+o(1))n$ for $n$ sufficiently large. In \cite{L}, {\L}uczak introduced a technique that applies the Szemer\'edi's Regularity Lemma to reduce the problem to show the existence of a large enough monochromatic  matching in a component.

This technique has become useful in determining the asymptotical value of Ramsey number of cycles. In 2007, {\L}uczak and Figaj \cite{LF} determined the 3-colour Ramsey numbers for paths and even cycles asymptotically. These results were strengthened: Kohayakawa, Simonovits and Skokan \cite{KSS} extended to long odd  cycles, Gy\'arf\'as, Ruszink\'o, S\'ark\"ozy and Szemer\'edi \cite{GRSS} extended to long paths, and Benevides and Skokan \cite{BS} extended to long even cycles. Jenssen and Skokan \cite{JS} proved that $R_{k}(C_{n})= 2^{k-1}(n-1)+1$ for every $k$ and sufficiently large odd $n$,  Day and Johnson \cite{DJ} showed that it does not hold for all $k$ and $n$.

It is natural to replace the underlying complete graph by a complete bipartite graph. Let $k\geq 2$ be an integer and given bipartite graphs $H_{1},...,H_{k}$. The bipartite Ramsey number $br(H_{1},...,H_{k})$ is the minimum integer $N$ such that any $k$-edge-coloring of $K_{N,N}$ contains a monochromatic $H_{i}$ in color $i$ for some $1\leq i\leq k$. The study of bipartite Ramsey number was initiated in the early 70s by Faudree and Schelp \cite{FS1} and independently by Gy\'arf\'as and Lehel \cite{GL} who showed that
    \[br_{2}(P_{n})= \left \{
    \begin{array}{ll}
      n-1 & \mbox {if $n$ is even,}\\
      n & \mbox {if $n$ is odd.}
    \end{array}
    \right.
    \]
Zhang and Sun \cite{ZS} determined $br(C_{2n},C_{4})=n+1$ for $n\geq 4$ and Zhang, Sun and Wu \cite{ZSW} determined the value of $br(C_{2n}, C_{6})$ for $n\geq 4$. Goddard, Henning and Oellermann \cite{GHO} determined $br(C_{4},C_{4},C_{4})=11$. Joubert \cite{J} showed that $$br(C_{2t_{1}},C_{2t_{2}},...,C_{2t_{k}})\leq k(t_{1}+t_{2}+...+t_{k}-k+1)$$ where $t_{i}$ is an integer and $2\leq t_{i}\leq 4$ for all $1\leq i\leq k$. Recently, Shen, Lin and Liu \cite{SLL} gave the asymptotic value of $br(C_{2n}, C_{2n})$. In \cite{LP} and \cite{LP1}, Liu and Peng gave the asymptotic value of $br(C_{2\lfloor \alpha_{1} n\rfloor}, C_{2\lfloor \alpha_{1} n\rfloor})$ and $br(C_{2\lfloor \alpha_{1} n\rfloor}, C_{2\lfloor \alpha_{1} n\rfloor},...,C_{2\lfloor \alpha_{r} n\rfloor})$ in some conditions. They also gave a minimum degree condition. The best known lower bound on $br_{2}(K_{n,n})$ is due to Hattingh and Henning \cite{HH} and the best known upper bound is due to Conlon \cite{C}.

\begin{defi}\label{def2.1}
  We say that $M$ is a $k$-connected matching in graph $G$, if $M$ is a component (a maximal connected subgraph) and the size of a maximum matching in $M$ is at least $k$.
 \end{defi}

 \begin{defi}\label{def1.1}
  Let $r(k,l,m)$ denote the smallest integer $n$ such that for any 3-coloring of $K_{n,n}$, there is a monochromatic $k$-connected matching in color 1 or a monochromatic $l$-connected matching in color 2 or a monochromatic $m$-connected matching in color 3.
 \end{defi}

   Buci\'c, Letzter and Sudakov \cite{BLS} have determined the exact value of $r(k,l,l)$ by applying K\"onig's theorem as a tool.

  \begin{theo}[Buci\'c, Letzter, Sudakov \cite{BLS}] \label{theo1.1}
    \[r(k,l,l)= \left \{
    \begin{array}{llll}
      k+2l-2 & \mbox {if $l \leq \frac{k+1}{2}$,}\\
      4l-2 & \mbox {if $\frac{k+1}{2} < l \leq \frac{2k}{3}$,}\\
      2k+l-2 & \mbox {if $\frac{2k}{3} < l < k$,}\\
      k+2l-2 & \mbox {if $k \leq l$.}
    \end{array}
    \right.
    \]
  \end{theo}

  Applying the above theorm and Regularity Lemma ({\L}uczak's technique), they have determined that $br(C_{2n},C_{2n}, C_{2n})=(3+o(1))n$ when $n$ is sufficiently large.  They commented that it is natural to consider the asymmetric 3-color bipartite Ramsey numbers for cycles, and  determining $r(k,l,m)$ is interesting in its own right.
  Applying K\"onig's theorem as in \cite{BLS}, we determine the exact value of $r(k,l,m)$ completely as stated in the following theorem.

  \begin{theo}\label{theo3}
    Let $2\leq k< l< m$, then we have
    \[r(k,l,m)=\left \{
    \begin{array}{lll}
     k+2m-2 & \mbox {if $3\leq k< l <m \leq l+\frac{k-1}{2}$,}\\
     2k+2l-3 & \mbox {if $3\leq k< l, l+\frac{k-1}{2}< m< k+l-1$,}\\
     k+l+m-2 &\mbox {if $2\leq k< l, m\geq k+l-1$.}
    \end{array}
    \right.
    \]
  \end{theo}

  Then applying the technique of {\L}uczak, we obtain the asymptotic value of \\  $br(C_{2\lfloor \alpha_{1}n\rfloor}, C_{2\lfloor \alpha_{2}n\rfloor}, C_{2\lfloor \alpha_{3}n\rfloor})$.

  \begin{theo}\label{theo2}
    Let $\alpha_{1}>0$, then
    \[br(C_{2\lfloor \alpha_{1}n\rfloor}, C_{2\lfloor \alpha_{2}n\rfloor}, C_{2\lfloor \alpha_{3}n\rfloor})=\left \{
    \begin{array}{lll}
      (\alpha_{1}+ 2\alpha_{3}+ o(1))n &\mbox {if $\alpha_{1}< \alpha_{2}< \alpha_{3}\leq \frac{\alpha_{1}}{2}+\alpha_{2}$,}\\
      (2\alpha_{1}+ 2\alpha_{2}+ o(1))n &\mbox {if $\alpha_{1}< \alpha_{2}, \frac{\alpha_{1}}{2}+\alpha_{2}< \alpha_{3}< \alpha_{1}+\alpha_{2}$,}\\
      (\alpha_{1}+ \alpha_{2}+ \alpha_{3}+ o(1))n &\mbox {if $\alpha_{1}< \alpha_{2}, \alpha_{3}\geq \alpha_{1}+\alpha_{2}$.}
    \end{array}
    \right.
    \]
  \end{theo}

  The organization of the paper is: In section 2, we give some preliminaries on definitions and useful facts. In section 3, we give the proof of $r(k,l)= k+l-1$ which mentioned in \cite{BLS} without a detailed proof. In section 4, we give the proof of Theorem \ref{theo3}.  As mentioned earlier, applying Theorem \ref{theo3} and Regularity lemma, we can obtain Theorem \ref{theo2}. This technique of {\L}uczak has become fairly standard in this area, for the completeness, we give the proof of Theorem \ref{theo2} in Appendix.

  \section{Preliminaries}

  We assume that bipartite graphs under consideration have bipartition $V_1\cup V_2$.
  \begin{defi}\label{def2.2}
   We say that a component is in $V_{i}$ if there exist a minimum vertex cover of this component in $V_{i}$.
  \end{defi}

  \begin{defi}\label{def2.3}
    We call a vertex a cover vertex of a component if there is a minimum vertex cover of this component containing it. Furthermore, if this component is monochromatic in color $i$, then we call a cover vertex of this component an $i$ cover vertex.
  \end{defi}

  \begin{remark}\label{remark2.1}
    If a component $C$ is in $V_{1}$, then $C\cap V_{1}$ is a minimum vertex cover of $C$ and each vertex in $C\cap V_{1}$ is a cover vertex. \q
  \end{remark}

  \begin{remark}\label{remark2.2}
    If $C$ is a $k$-connected matching, then $|C\cap V_{1}|\geq k$ and $|C\cap V_{2}|\geq k$. Moreover, if no minimum vertex cover of $C$ is in $V_{1}$, then $|C\cap V_{1}|\geq k+1$. \q
  \end{remark}

  \begin{defi}\label{def2.4}
    We call a vertex a red-blue vertex, if this vertex belongs to the intersection of a red component and a blue component.
  \end{defi}

  The low bound of $r(k_{1}, k_{2},..., k_{p})$ can be easily obtained by constructing:

  Let $n= \sum_{i=1}^p k_{i}-p$ and $K_{n,n}$ with vertex set $V_{1}\cup V_{2}$. Partition $V_{1}$ into $p$ parts $\cup ^{p}_{i=1}S_{i}$ with $|S_{i}|= k_{i}-1$. Coloring all edges between $S_{i}$ and $V_{2}$ in $i$-th color . It is easy to see that there is no monochromatic $k_{i}$-connected matching for any $i\in [1,p]$.

  \begin{fact}\label{fact2.2}
    $r(k_{1}, k_{2},..., k_{p})\geq \sum_{i=1}^p k_{i}-p+1$. \q
  \end{fact}

  One of the most important tools in this paper is K\"onig's theorem:

  \begin{theo}\label{theo1}
    In a bipartite graph, the size of a maximum matching is equal to the number of vertices in a minimum vertex cover.\q
  \end{theo}

  \section{The exact value of r(k,l)}

  \begin{fact}\label{fact2.3}
    $r(2,2)=3$.
  \end{fact}

  \begin{proof}
    By Fact \ref{fact2.2}, $r(2,2) \geq 3$, so we need to show that $r(2,2) \leq 3$. If $r(2,2)\neq 3$, then there exist a 2-edge-coloring of $K_{3,3}$ such that there is no monochromatic 2-connected matching. Since every vertex has degree three, every vertex is incident to at least two edges in the same color. W.L.O.G, suppose that $v\in V_{1}$ is incident to two red edges $vu_{1}, vu_{2}$. Then the edges connecting $u_{1}$, $u_{2}$ with others vertices in $V_{1}$ must be blue for there is no red 2-connected matching. But this generates a blue 2-connected matching. A contradiction. So we have shown that $r(2,2)\leq 3$. \q
  \end{proof}

  \begin{lemma}\label{lem2.4}
    $r(k,k)=2k-1$ for any $k\geq 2$.
  \end{lemma}

  \begin{proof}
    By Fact \ref{fact2.2}, $r(k,k) \geq 2k-1$, so we need to show that $r(k,k) \leq 2k-1$.

    Use induction on $k$. The assertion holds for $k=2$ by Fact \ref{fact2.3}. Assume that $k \geq 2$ and the assertion holds for $k$. If $r(k+1,k+1) \neq 2k+1$, then there exist a 2-edge-coloring of $K_{2k+1,2k+1}$ such that there is no monochromatic $(k+1)$-connected matching. Note that it contains at most two red $k$-connected matchings $\{R_{1}, R_{2}\}$ and at most two blue $k$-connected matchings $\{B_{1}, B_{2}\}$. Clearly, the size of a maximum matching in  these monochromatic $k$-connected matchings is exactly $k$.

    Case 1: There are exactly two red $k$-connected matchings and two blue $k$-connected matchings.

    By Remark \ref{remark2.2}, each $k$-connected matching not in $V_{i}$ contains at least $k+1$ vertices in $V_{i}$ and $|V_{i}|=2k+1$, then $V_{i}$ contains at least one red $k$-connected matching and one blue $k$-connected matching $(i=1, 2)$.

    Subcase 1.1: Four monochromatic $k$-connected matchings are in $V_{1}$, then it is easy to show that $R_{1}\cap B_{1}\cap V_{1}\neq \emptyset, R_{2}\cap B_{2}\cap V_{1}\neq \emptyset$ (interchange the subscripts of $B_{i}'s$ if necessary). By the assumption, every vertex in $R_{i} \cap V_{1}(B_{j}\cap V_{1})$ is a red (blue) cover vertex of $R_{i}(B_{j})$. Then remove one red-blue cover vertex from $R_{1}\cap B_{1}\cap V_{1}$ and $R_{2}\cap B_{2}\cap V_{1}$ respectively and two vertices in $V_{2}$. Then the cardinality of a minimum vertex cover of each $R_{i},B_{j}$ is reduced by one. By K\"{o}nig's theorem, after the removal of these four vertices, there is neither red $k$-connected matching nor blue $k$-connected matching in the remaining $K_{2k-1,2k-1}$. A contradiction to our induction hypothesis that $r(k,k)=2k-1$.

    Subcase 1.2: Three monochromatic $k$-connected matchings are in $V_{1}$ and another is not in $V_{1}$.

    W.L.O.G, suppose two red $k$-connected matchings $\{R_{1}, R_{2}\}$ and one blue $k$-connected matching $B_{1}$ are in $V_{1}$. Again, every vertex of these $k$-connected matchings in $V_{1}$ is a cover vertex of these $k$-connected matchings. Note that $B_{1}$ must intersect with at least one red $k$-connected matching (assume that is $R_{1}$) in $V_{1}$. Remove one cover vertex from $R_{1}\cap B_{1}\cap V_{1}$, one cover vertex of $R_{2}\cap V_{1}$, and remove one cover vertex from $B_{2}\cap V_{2}$, one of any other vertices in $V_{2}$. Then there is neither red $k$-connected matching nor blue $k$-connected matching in the remaining $K_{2k-1,2k-1}$. A contradiction again.

    Subcase 1.3: Two $k$-connected matchings are in $V_{1}$ and others are not in $V_{1}$, then remove one cover vertex from two $k$-connected matchings in $V_{1}$ respectively and one cover vertex from other two $k$-connected matchings in $V_{2}$ respectively. A contradiction again.

    Case 2: There are two red $k$-connected matchings and one blue $k$-connected matching. (The argument is the same if there are one red $k$-connected matching and two blue $k$-connected matchings.) In this case, by Remark \ref{remark2.2} and $|V_{i}|=2k+1$ $(i=1, 2)$, at least one red $k$-connected matching is in $V_{1}$ and at least one red $k$-connected matching is in $V_{2}$.

    Subcase 2.1: All of them are in $V_{1}$. It is similar to subcase 1.2 which we have discussed.

    Subcase 2.2: One red $k$-connected matching is in $V_{1}$ (say $R_{1}$) and at least one of the other two monochromatic $k$-connected matchings (say $R_{2}$) has at least 1 cover vertex in $V_{2}$. Then remove one cover vertex from $R_{1}\cap V_{1}$ and one cover vertex from $R_{2}\cap V_{2}$. Remove one cover vertex from the blue $k$-connected matching and one more vertex in $V_{1}$ or $V_{2}$ such that two vertices have been removed from each of $V_{i}$ ($i=1,2$). A contradiction again.

    Case 3: There are two monochromatic $k$-connected matchings.
    Remove one cover vertex from each of these two monochromatic $k$-connected matchings. Remove other vertices from $V_{i}$ ($i=1,2$) until two vertices have been removed from each of $V_{i}$ ($i=1,2$). A contradiction again.

    Case 4: There is one monochromatic $k$-connected matching.
    Remove one cover vertex from this monochromatic $k$-connected matching. Remove other vertices from $V_{i}$ ($i=1,2$) until two vertices have been removed from each of $V_{i}$ ($i=1,2$). A contradiction again.

    There is at least one monochromatic $k$-connected matching, so we have discussed all possible cases and have shown that $r(k, k) \leq 2k-1$. \q
  \end{proof}

  \begin{lemma}\label{lem2.5}
    $r(k,k+1)=2k$ for $k\geq 2$.
  \end{lemma}

  \begin{proof}
    By Fact \ref{fact2.2}, $r(k,k+1) \geq 2k$, so we need to show that $r(k,k+1) \leq 2k$. If $r(k,k+1) \neq 2k$, then there exist a 2-edge-coloring of $K_{2k,2k}$ such that there is no red $k$-connected matching and at most two blue $k$-connected matchings. Clearly, the size of a maximum matching in these monochromatic $k$-connected matchings is exactly $k$.

    Case 1: There are two blue $k$-connected matchings. Then each of two blue $k$-connected matching must have $k$ vertices in each part of $K_{2k,2k}$. Note that each vertex of a blue $k$-connected matching is a cover vertex. Then remove one cover vertex of a blue $k$-connected matching in $V_{1}$ and one cover vertex of another blue $k$-connected matching in $V_{2}$. By K\"onig's theorem there is neither red $k$-connected matching nor blue $k$-connected matching in the remaining $K_{2k-1,2k-1}$. A contradiction to Lemma \ref{lem2.4}.

    Case 2: There is one blue $k$-connected matching (say $B$). Then remove one cover vertex from $B\cap V_{i}$ $(i=1$ or $2)$ and any other vertex from $V_{3-i}$. A contradiction again.\q
  \end{proof}

  \begin{theo}\label{theo2.1}
    $r(k,l)=k+l-1$ for $l\geq k\geq 2$.
  \end{theo}

  \begin{proof}
    By Fact \ref{fact2.2}, $r(k,l)\geq k+l-1$, so we need to show that $r(k,l)\leq k+l-1$.

    Use induction on $l$. The assertion holds for $l=k$ by Lemma \ref{lem2.4} and for $l=k+1$ by Lemma \ref{lem2.5}. Assume that $l\geq k+1$ and the assertion holds for $l$. If $r(k,l+1) \neq k+l$, then there exist a 2-edge-coloring of $K_{k+l,k+l}$ such that there is no red $k$-connected matching and at most one blue $l$-connected matching. Clearly, the size of a maximum matching in the blue $k$-connected matching is exactly $l$.
    
    Remove one cover vertex from this blue $l$-connected matching in $V_{i}$ $(i=1,2)$ and remove any other vertex in $V_{3-i}$. Then by K\"onig's theorem, there is neither red $k$-connected matching nor blue $l$-connected matching in the remaining $K_{k+l-1,k+l-1}$. A contradiction to our induction hypothesis. So $r(k,l+1) \leq k+l$ when $l\geq k+1$. The proof is complete. \q
  \end{proof}

  \section{Exact value of r(k,l,m)}

  In this section, we will determine the exact value of $r(k,l,m)$. The exact value of $r(k,k,k)$ and $r(k,l,l)$ have been determined in \cite{BLS}. So we assume that $m >l >k\geq 2$ in this section.

  \begin{lemma}\label{lemma5.2}
    $r(k,l,m)\geq k+2m-i-2$ for $0\leq i\leq k-2, 3\leq k< l< m\leq \frac{k+i-1}{2}+l$.
  \end{lemma}

  \begin{proof}
    Let $k,l,m,i$ satisfy the conditions. We show that there exist a 3-edge-coloring of $K_{k+2m-i-3,k+2m-i-3}$ such that there is no monochromatic $k$-connected matching, no monochromatic $l$-connected matching and no monochromatic $m$-connected matching. Then we have $r(k,l,m)\geq k+2m-i-2$.

    In $K_{k+2m-i-3,k+2m-i-3}$, partition $V_{1}$ into 3 sets: $S_{1}, S_{2}, S_{3}$ with $\mid S_{1}\mid=m-1$, $\mid S_{2}\mid=m-1$, $\mid S_{3}\mid=k-1-i$. Partition $V_{2}$ into 4 sets: $T_{1}, T_{2}, T_{3}, T_{4}$. If $2(m-l)-i> 0$, then let $\mid T_{1}\mid=l-1$, $\mid T_{2}\mid=l-1$, $\mid T_{3}\mid=k-1$, $\mid T_{4}\mid=2(m-l)-i$. Else let $\mid T_{1}\mid \leq l-1$, $\mid T_{2}\mid \leq l-1$, $\mid T_{3}\mid \leq k-1$, $\mid T_{4}\mid= 0$ ($k+2l-3\geq k+2m-i-3$ since $2(m-l)-i\leq 0$). It is easy to see that $2(m-l)-i\leq k-1$ since $l< m\leq \frac{k+i-1}{2}+l$. Let C: $(i,j)$ represent that all of the edges between $S_{i}$ and $T_{j}$ are colored C. We colour $K_{k+2m-i-3,k+2m-i-3}$ as follows:

    Red: $(1,4), (2,3), (3,1), (3,2)$;

    Blue: $(1,2), (2,1), (3,3), (3,4)$;

    Green: $(1,1), (1,3), (2,2), (2,4)$.

    It is easy to verify that there is no red $k$-connected matching, no blue $l$-connected matching and no green $m$-connected matching in this 3-coloring. This shows that\\ $r(k,l,m)\geq k+2m-i-2$.\q
  \end{proof}

  We will show that the lower bound is also an upper bound.

  \begin{lemma}\label{lemma5.1}
    $r(k,l,l+1)\leq k+2l$ for $l> k\geq 2$.
  \end{lemma}

  \begin{proof}
    By Theorem \ref{theo1.1}, we have $r(k,l,l)=k+2l-2$ for $l> k$.
    If $r(k,l,l+1)\neq k+2l$, then there exist a 3-edge-coloring of $K_{k+2l,k+2l}$ such that there is no red $k$-connected matching, no blue $l$-connected matching and at most two green $l$-connected matchings. Clearly, the size of maximum matching in these green $l$-connected matchings is exactly $l$.

    If there are two green $l$-connected matchings, then remove one cover vertex from each of these two green $l$-connected matchings and remove other vertices from $V_{i}$ ($i=1,2$) until two vertices have been removed from each of $V_{i}$ ($i=1,2$). Then there is no red $k$-connected matching, no blue $l$-connected matching and no green $l$-connected matching in the remaining $K_{k+2l-2, k+2l-2}$. A contradiction to $r(k,l,l)=k+2l-2$.

    If there is one green $l$-connected matching, then remove one cover vertex from this green $l$-connected matching and remove other vertices from $V_{i}$ until two vertices have been removed from each of $V_{i}$ $(i=1,2)$. A contradiction again.

   So we have shown that $r(k,l,l+1)\leq k+2l$ for $l> k\geq 2$.\q
  \end{proof}

  \begin{theo}\label{theo5.1}
    $r(k,l,m)=k+2m-2$ for $3\leq k< l< m\leq l+\frac{k-1}{2}$.
  \end{theo}

  \begin{proof}
    Let $i=0$ in Lemma \ref{lemma5.2}, then we have $r(k,l,m)\geq k+2m-2$ when $3\leq k< l< m\leq l+ \frac{k-1}{2}$. So what left is to prove that $r(k,l,m)\leq k+2m-2$ for $3\leq k< l< m\leq l+ \frac{k-1}{2}$. In fact, we will prove a slightly stronger result: $r(k,l,m)\leq k+2m-2$ for $2\leq k< l< m$.

    Use induction on $m$. The assertion holds for $m=l+1$ by Lemma \ref{lemma5.1}. Assume $m\geq l+1$ and the assertion holds for $m$. If $r(k,l,m+1)\neq k+2m$, then there exist a 3-edge-coloring of $K_{k+2m,k+2m}$ such that there is no red $k$-connected matching, no blue $l$-connected matching and at most two green $m$-connected matchings. Clearly, the size of maximum matching in these green $m$-connected matchings is exactly $m$.

    If there are two green $m$-connected matchings, then remove one cover vertex from each of these two green $m$-connected matchings and remove other vertices from $V_{i}$ ($i=1,2$) until two vertices have been removed from each of $V_{i}$ ($i=1,2$). Now there is no red $k$-connected matching, no blue $l$-connected matching and no green $m$-connected matching in the remaining $K_{k+2m-2, k+2m-2}$. A contradiction to our induction hypothesis.

    If there is one green $m$-connected matching, then remove one cover vertex from this green $m$-connected matching and remove other vertices from $V_{i}$ ($i=1,2$) until two vertices have been removed from each of $V_{i}$ ($i=1,2$). A contradiction again.

    So we have shown that $r(k,l,m)\leq k+2m-2$ for $2\leq k< l< m$.\q
  \end{proof}

  \begin{lemma}\label{lemma5.3}
    $r(k,l,k+l-1)=2k+2l-3$ for $l> k\geq 2$.
  \end{lemma}

  \begin{proof}
    By Fact \ref{fact2.2}, $r(k,l,k+l-1)\geq 2k+2l-3$, so we need to show that $r(k,l,k+l-1)\leq 2k+2l-3$. If $r(k,l,k+l-1)\neq 2k+2l-3$, then there exist a 3-edge-coloring of $K_{2k+2l-3,2k+2l-3}$ such that there is no red $k$-connected matching, no blue $l$-connected matching and no green $(k+l-1)$-connected matching. We will show that it contradicts to $r(k,l,l)=k+2l-2$ guaranteed by Theorem \ref{theo1.1} when $l> k\geq 2$.

    Note that there are at most three green components which contain matchings with size in $[l,k+l-2]$ because $4l> 2k+2l-3$.

    Case 1: There is one green component whose size of a maximum matching is $i$ for some $i\in [l,k+l-2]$. We can remove $k-1$ cover vertices in $G$ and remove other vertices in $V_{1}\cup V_{2}$ such that $k-1$ vertices have been removed from each of $V_{i}$ $(i=1,2)$. Now, $G$ is at most a $(l-1)$-connected matching because $i-(k-1)\leq (k+l-2)-(k-1)=l-1$. So there is no red $k$-connected matching, no blue $l$-connected matching and no green $l$-connected matching in the remaining $K_{k+2l-2,k+2l-2}$. A contradiction to $r(k,l,l)=k+2l-2$.

    Case 2: There are one green component whose size of a maximum matching is $i$ (say $G_{1}$) and one green component whose size of a maximum matching is $j$ (say $G_{2}$), where $l\leq i\leq j\leq k+l-2$.

    Subcase 2.1: $G_{1}$ has at most $k+l-2$ vertices in both $V_{1}$ and $V_{2}$. By pigeonhole principle, $G_{2}$ has at least $\lceil \frac{j}{2} \rceil$ cover vertices in $V_{1}$ or $V_{2}$. W.L.O.G, assume that $G_{2}$ has at least $\lceil \frac{j}{2} \rceil$ cover vertices in $V_{1}$. Then we remove $p= \min\{\lceil \frac{j}{2} \rceil,k-1\}$ cover vertices from $G_{2}\cap V_{1}$ and $k-1-p$ other vertices from $V_{1}$, remove $k-1$ vertices from $G_{1}\cap V_{2}$. Now, what remains in $G_{2}$ is not a green $l$-connected matching because of the following: the number of cover vertices left in $G_{2}$ is $j-(k-1)\leq (k+l-2)-(k-1)=l-1$ if $p=k-1$ or if $p=\lceil \frac{j}{2}\rceil$, then the number of cover vertices left in $G_{2}$ is at most $\lfloor \frac{j}{2}\rfloor \leq \lceil \frac{j}{2}\rceil \leq k-1 \leq l-1$. And the number of remaining vertices in $G_{1}\cap V_{2}$ is at most $(k+l-2)-(k-1)=l-1$. Hence there is no red $k$-connected matching, no blue $l$-connected matching and no green $l$-connected matching in the remaining $K_{k+2l-2,k+2l-2}$. A contradiction again.

    Subcase 2.2: $G_{1}$ has at least $k+l-1$ vertices in $V_{1}$ or $V_{2}$. W.L.O.G assume that $G_{1}$ has at least $k+l-1$ vertices in $V_{1}$. So $G_{2}$ has at most $k+l-2$ vertices in $V_{1}$.

    If $G_{1}$ has at most $k+l-2$ vertices in $V_{2}$. Then remove $k-1$ vertices from $G_{2}\cap V_{1}$ and remove $k-1$ vertices from $G_{1}\cap V_{2}$. Now in the remaining $K_{k+2l-2,k+2l-2}$, the size of a maximum matching in $G_{1}$ and $G_{2}$ is at most $(k+l-2)-(k-1)=l-1$. A contradiction again.

    Else $G_{1}$ has at least $k+l-1$ vertices in $V_{2}$. Then $G_{2}$ has at most $k+l-2$ vertices in $V_{2}$. It is the same as subcase 2.1.

    Case 3: There are one green component whose size of a maximum matching is $i$ (say $G_{1}$),   one green component whose size of a maximum matching is $j$ (say $G_{2}$), and one green component whose size of a maximum matching is $t$ (say $G_{3}$), where $l\leq i\leq j\leq t\le k+l-2$.

    Note that $l\leq 2k-3$ in this case because $3l> 2k+2l-3$ if $l> 2k-3$ and note that $l\leq |G_{i}\cap V_{j}|\leq 2k-3$ $(i=1,2,3; j=1,2)$. Also note that $t\geq l$ and $2k+2l-3-t\leq 2k+l-3$, then $G_{1}\cup G_{2}$ has at most $2k+l-3$ vertices in both $V_{1}$ and $V_{2}$. 
    
    By pigeonhole principle, $G_{3}$ has at least $\lceil \frac{t}{2} \rceil$ cover vertices in $V_{1}$ or $V_{2}$. W.L.O.G,  assume that $G_{3}$ has at least $\lceil \frac{t}{2} \rceil$ cover vertices in $V_{1}$. Then remove $k-1$ vertices from $G_{3}\cap V_{1}$ such that they contain as many cover vertices $(\min \{\lceil \frac{t}{2} \rceil, k-1\})$ of $G_{3}$ as possible. Remove $k-1$ vertices from $V_{2}$ such that $G_{1}$ has at most $l-1$ vertices in $V_{2}$ and $G_{2}$ has at most $l-1$ vertices in $V_{2}$.  This is possible due to the following reason: If $G_{1}$ has $s$ vertices in $V_{2}$ (recall that $l\leq s\leq 2k-3$), then remove $s-(l-1)$ vertices from $G_{1}\cap V_{2}$ and $(k-1)-(s-l+1)$ vertices from $G_{2}\cap V_{2}$. Then $G_{2}$ has at most $(2k+l-3)-s-[(k-1)-(s-l+1)]=k-1\leq l-1$ vertices in $V_{2}$ in the remaining $K_{k+2l-2,k+2l-2}$. $G_{1}$ has at most $\max \{(2k-3)-(k-1), \lfloor \frac{t}{2}\rfloor \}\leq k-1\leq l-1$ cover vertices in the remaining $K_{k+2l-2, k+2l-2}$. Now there is no red $k$-connected matching, no blue $l$-connected matching and no green $l$-connected matching in the remaining $K_{k+2l-2,k+2l-2}$. A contradiction again.

    So we have shown that $r(k,l,k+l-1)\leq 2k+2l-3$ for $l> k\geq 2$.\q
  \end{proof}

  \begin{theo}\label{theo5.4}
    $r(k,l,m)=2k+2l-3$ for $3\leq k< l, l+\frac{k-1}{2}< m< k+l-1$.
  \end{theo}

  \begin{proof}
    Actually, by Lemma \ref{lemma5.2}, we have $r(k,l,m)\geq k+2m-i-2$ for $1\leq i\leq k-2, 3\leq k< l, \frac{k+i-2}{2}+l< m\leq \frac{k+i-1}{2}+l$. Let $k,l,m,i$ satisfy the condition. Note that $2k+2l-4< k+2m-i-2\leq 2k+2l-3$. Since $k,l,m,i$ are integers, then $k+2m-i-2=2k+2l-3$. So we have $r(k,l,m)\geq k+2m-i-2=2k+2l-3$. By Lemma \ref{lemma5.3}, we have $r(k,l,k+l-1)=2k+2l-3$. Since $m< k+l-1$, then $r(k,l,m)\leq r(k,l,k+l-1)= 2k+2l-3$.\q
  \end{proof}

  \begin{theo}\label{theo5.2}
    $r(k,l,m)=k+l+m-2$ for $2\leq k<l ,m\geq k+l-1$.
  \end{theo}

  \begin{proof}
    By Fact \ref{fact2.2}, $r(k,l,m)\geq k+l+m-2$, so we need to show that $r(k,l,m)\leq k+l+m-2$ for $2\leq k< l, m\geq k+l-1$.
    Using induction on $m$. The assertion holds for $m=k+l-1$ by Lemma \ref{lemma5.3}. Assume that $m\geq k+l-1$ and the assertion holds for $m$. If $r(k,l,m+1)\neq k+l+m-1$, then there exist a 3-edge-coloring of $K_{k+l+m-1,k+l+m-1}$ such that there is no red $k$-connected matching, no blue $l$-connected matching and at most two green $m$-connected matchings.

    Case 1: There is one green $m$-connected matching (say $G$). Then remove one cover vertex from $G$ in $V_{i}$ $(i=1,2)$ and any other vertex in $V_{3-i}$. Now, there is no red $k$-connected matching, no blue $l$-connected matching and no green $m$-connected matching in the remaining $K_{k+l+m-2,k+l+m-2}$. A contradiction to our induction hypothesis.

    Case 2: There are two green $m$-connected matchings (say $G_{1}, G_{2}$). Then $2m\leq k+l+m-1$ which shows that $m=k+l-1$. So each of $G_{i}$ has exactly $k+l-1$ vertices in both $V_{1}$ and $V_{2}$ and each vertex in $G_{i}\cap V_{j}$ is a cover vertex for $G_{i}$ $(1\leq i,j \leq 2)$. Remove one cover vertex from $G_{1}\cap V_{1}$ and one cover vertex from $G_{2}\cap V_{2}$. A contradiction again.

    So we have shown that $r(k,l,m)\leq k+l+m-2$ for $2\leq k< l, m\geq k+l-1$.\q
  \end{proof}

  Combining Theorem \ref{theo5.1}, \ref{theo5.4}, \ref{theo5.2}, we obtain Theorem \ref{theo3}.


  \section{Appendix: Proof of Theorem \ref{theo2}}

  \subsection{Monochromatic connected matchings in almost complete bipartite graphs}

  The main result in this section is to extend Theorem \ref{theo3} to an almost complete bipartite graph. We just give the proof for the case $\alpha_{1}< \alpha_{2}, \alpha_{3}\geq \alpha_{1}+\alpha_{2}$, two other cases can be proven in the same way.

  \begin{theo}\label{theo4.1}
    Let $0< \alpha_{1}< \alpha_{2}, \alpha_{3}\geq \alpha_{1}+\alpha_{2}, \beta=(2\lceil \frac{\alpha_{1}+ \alpha_{2}+\alpha_{3}+1}{\alpha_{1}} \rceil)^6,\gamma= \frac{12}{\alpha_{1}} \lceil \frac{\alpha_{1}+ \alpha_{2}+ \alpha_{3}+ 1}{\alpha_{1}} \rceil$. For every $0< \varepsilon <\frac{1}{\beta+ (\alpha_{1}+ \alpha_{2}+\alpha_{3})\gamma}$, there is $n_{0} = n_{0}(\varepsilon)$ such that the following holds. For $n> n_{0}$, let $G$ be a bipartite graph with partition $\{V_{1}, V_{2}\}$ and $|V_{1}|=|V_{2}|=N$, where $N\geq (\alpha_{1}+ \alpha_{2}+\alpha_{3}+(\beta+ (\alpha_{1}+ \alpha_{2}+\alpha_{3})\gamma)\varepsilon)n$. Suppose that every vertex in $V_{1}$ has at most $\varepsilon n$ non-neighbours in $V_{2}$ and vice versa. Then for every red-blue-green-edge-coloring of $G$, there is a red $\lfloor \alpha_{1}n \rfloor$-connected matching or a blue $\lfloor \alpha_{2}n \rfloor$-connected matching or a green $\lfloor \alpha_{3}n \rfloor$-connected matching.
  \end{theo}

  The proof follows from \cite{BLS}. The idea is to add the non-edge to $G$ such that it becomes a complete bipartite graph, then apply Theorem \ref{theo3}. Let $N= \lceil (\alpha_{1}+ \alpha_{2}+\alpha_{3}+(\beta+ (\alpha_{1}+ \alpha_{2}+\alpha_{3})\gamma)\varepsilon)n \rceil$.

  \begin{defi}\label{def4.1}
    We call $C_{R,i}$ a red virtual component if $C_{R,i}$ is a red component of order at least $\alpha_{1}n$ or a maximal union of red components with order no more than $2\alpha_{1}n$ (the order of each of components is no more than $\alpha_{1}n$). Define blue virtual component and green virtual component in the same way.
  \end{defi}

  \begin{remark}\label{remark4.1}
    It is obvious that each virtual component has no intersection with other virtual components in the same color. The maximum number of virtual components in $G$ is at most $2\lceil \frac{\alpha_{1}+ \alpha_{2}+\alpha_{3}+1}{\alpha_{1}} \rceil$ in each of colors.
  \end{remark}

  \begin{proof}
    W.L.O.G, consider red virtual components. It is easy to know that all but at most one of the virtual components have order at least $\alpha_{1}n$. So the maximum number of red virtual components in $G$ is at most $2\times \lceil \frac{(\alpha_{1}+ \alpha_{2}+\alpha_{3}+(\beta+ (\alpha_{1}+ \alpha_{2}+\alpha_{3})\gamma )\varepsilon )n}{\alpha_{1}n}\rceil \leq 2\lceil \frac{\alpha_{1}+ \alpha_{2}+\alpha_{3}+1}{\alpha_{1}} \rceil$ by the choice of $\varepsilon$. Similarly, the maximum number of blue (green) virtual components in $G$ is at most $2\lceil \frac{\alpha_{1}+ \alpha_{2}+\alpha_{3}+1}{\alpha_{2}} \rceil \leq 2\lceil \frac{\alpha_{1}+ \alpha_{2}+\alpha_{3}+1}{\alpha_{1}} \rceil$ ($2\lceil \frac{\alpha_{1}+ \alpha_{2}+\alpha_{3}+1}{\alpha_{3}} \rceil \leq 2\lceil \frac{\alpha_{1}+ \alpha_{2}+\alpha_{3}+1}{\alpha_{1}} \rceil$).
  \end{proof}

  \begin{defi}\label{def4.2}
    We call a non-edge bad if it is not contained in any virtual component.
  \end{defi}

  \begin{lemma}\label{lemma4.1}
    There is a set of at most $\beta \varepsilon n$ vertices that cover all bad non-edges in $G$ in each $V_{i}$ $(i=1,2)$.
  \end{lemma}

  \begin{proof}
    Each bad edge can be represented by Type $(a,b,c,d,e,f)$ ($a\neq d, b\neq e, c\neq f; a,b,c,d,e,f\in [1,2\lceil \frac{\alpha_{1}+ \alpha_{2}+\alpha_{3}+1}{\alpha_{1}} \rceil]$) such that one of its ends belongs to $C_{R,a}\cap C_{B,b}\cap C_{G,c}= U$ and another end belongs to $C_{R,d}\cap C_{B,e}\cap C_{G,f}= W$. (Recall that $\beta=(2\lceil \frac{\alpha_{1}+ \alpha_{2}+\alpha_{3}+1}{\alpha_{1}} \rceil)^6$ and there are $\beta$ choices for $(a,b,c,d,e,f)$. ) Now for each fixed $(a,b,c,d,e,f)$, we claim that $|U|, |W|\leq \varepsilon n$. Otherwise, since each vertex has at most $\varepsilon n$ non-neighbors, then there is an edge between $U$ and $W$. W.L.O.G, assume that it is in red (same as blue or green), then $C_{R,a}$ and $C_{R,d}$ are not disconnected in red, a contradiction to Remark \ref{remark4.1}. So we have shown that all bad non-edges of $(a,b,c,d,e,f)$ can be covered by a set ($U\cup W$) with at most $\varepsilon n$ vertices in each $V_{i}$ $(i=1,2)$.\q
  \end{proof}

  Take a virtual component $C$ and a minimum vertex cover $W$ in $C$. Now add non-edges to $G$ which are incident with $W$ inside $C$ and colour them by the same color as $C$. Repeat it until no non-edge can be added and denote the resulting graph by $G_{1}$. It is easy to see that the cardinality of a minimum cover of $C$ in $G_{1}$ is the same as in $G$, and by K\"onig's theorem, the size of the maximum matching in each virtual component in $G_{1}$ is the same as in $G$.

  A pair of vertices not an edge in $G_{1}$ is called a missing edge. The next lemma says that there is not too much pairwise disjoint missing edges in each virtual components in $G_{1}$.

  \begin{lemma}\label{lemma4.2}
    Let $C_{R}$ $(C_{B}, C_{G})$ be a red (blue, green) virtual component in $G_{1}$, and let $M$ be a matching of missing edges spanned by $C_{R}$ $(C_{B}, C_{G})$. Then $M$ contains at most $\gamma \varepsilon \alpha_{1} n$ $(\gamma \varepsilon \alpha_{2} n, \gamma \varepsilon \alpha_{3} n)$ missing edges.
  \end{lemma}

  \begin{proof}
    W.L.O.G, assume that the virtual component in $G_{1}$ is red. Let $\{x_{1}y_{1},...,x_{t}y_{t}\}$ be a matching of missing edges spanned by $C_{R}$ and suppose that $t\geq \gamma \varepsilon \alpha_{1} n$. Since $x_{i}y_{i}$ is missing, neither $x_{i}$ nor $y_{i}$ is in a minimum cover $W_{R}$ of $C_{R}$. Otherwise, we would have added $x_{i}y_{i}$ to $G_{1}$. So no red edges are spanned by $\{x_{1},...,x_{t},y_{1},...,y_{t}\}$. Otherwise, assume that $x_{i}y_{j}$ is a red edge, then $x_{i}y_{j}$ is not covered by $W_{R}$. A contradiction to that $W_{R}$ is a minimum cover of $C_{R}$.

    Since there are at most $2\lceil \frac{\alpha_{1}+ \alpha_{2}+\alpha_{3}+1}{\alpha_{1}} \rceil$ blue virtual components, there exist a blue virtual component $C_{B}$ that contains at least $\frac{t}{2\lceil \frac{\alpha_{1}+ \alpha_{2}+\alpha_{3}+1}{\alpha_{1}} \rceil}$ of $x_{i}$'s. Let $s = \frac{t}{2\lceil \frac{\alpha_{1}+ \alpha_{2}+\alpha_{3}+1}{\alpha_{1}} \rceil}$. Suppose that $x_{1},...,x_{s}\in C_{B}$.

    If there is at least half of the vertices $y_{1},...,y_{s}$ are in $C_{B}$. Assume $Y=\{y_{1},...,y_{\frac{s}{2}}\}$ is contained in $C_{B}$ and let $X=\{x_{1},...,x_{\frac{s}{2}}\}$. As discussed above, for a minimum cover $W_{B}$ of $C_{B}$, $x_{i}$ and $y_{i}$ are not in $W_{B}$ which implies that $X\cup Y$ spans no blue edges.

    Else there is at least half of the vertices $y_{1},...,y_{s}$ which are not in $C_{B}$. Suppose that the set $Y=\{y_{1},...,y_{\frac{s}{2}}\}$ is disjoint from $C_{B}$ and let $X=\{x_{1},...,x_{\frac{s}{2}}\}$. Then there are no blue edges between $X$ and $Y$.

    In each case, $X\cup Y$ spans neither red edge nor blue edge. Since $\frac{s}{2}\geq \frac{3\varepsilon \alpha_{1} n}{\alpha_{1}} = 3\varepsilon n$ and each vertex has at least $\frac{s}{2}-\varepsilon n \geq 2\varepsilon n$ neighbors (recall that $\gamma= \frac{12}{\alpha_{1}} \lceil \frac{\alpha_{1}+ \alpha_{2}+ \alpha_{3}+ 1}{\alpha_{1}} \rceil$), then $G_{1}[X,Y]$ is connected in green. Let $C_{G}$ be the green virtual component containing $X\cup Y$. Since every vertex in $X\cup Y$ is incident with a missing edge spanned by $C_{G}$ (the structure of $G_{1}$), it follows that none of the vertices in $X\cup Y$ is in a minimum cover $W_{G}$ of $C_{G}$. Hence, $X\cup Y$ cannot span any green edge. A contradiction.\q
  \end{proof}

  For each missing edge in $G_{1}$ that is not bad, take a virtual component containing it and add the edge to $G_{1}$ in color of the chosen component. Denote the resulting graph by $G_{2}$. Now, we are ready to prove the main result in this section.\\

  \noindent \emph{Proof of Theorem \ref{theo4.1}:}
  Let $W$ be a set of vertices that cover all bad edges with the same number of vertices on both sides. By Lemma \ref{lemma4.1}, $W$ has size at most $2\beta \varepsilon n$. Let $G_{3}=G_{2}\setminus W$. Then $G_{3}$ is a 3-edge-colored complete bipartite graph with at least $N-\beta \varepsilon n$ vertices on each side. By Theorem \ref{theo5.2}, $G_{3}$ contains a red $((1+\gamma \varepsilon)\alpha_{1}n)$-connected matching since $\frac{\alpha_{1}}{\alpha_{1}+\alpha_{2}+\alpha_{3}} \times (N-\beta \varepsilon n)\geq (1+\gamma \varepsilon)\alpha_{1}n$, or a blue $((1+\gamma \varepsilon)\alpha_{2}n)$-connected matching since $\frac{\alpha_{2}}{\alpha_{1}+\alpha_{2}+\alpha_{3}} \times (N-\beta \varepsilon n)\geq (1+\gamma \varepsilon)\alpha_{2}n$, or a green $((1+\gamma \varepsilon)\alpha_{3}n)$-connected matching since $\frac{\alpha_{3}}{\alpha_{1}+ \alpha_{2}+\alpha_{3}} \times (N-\beta \varepsilon n)\geq (1+\gamma \varepsilon)\alpha_{3}n$. W.L.O.G, assume that $G_{3}$ contains a red $((1+\gamma \varepsilon) \alpha_{1}n)$-connected matching $M$. By the construction of $G_{2}$, $M$ is contained in a red virtual component $C_{R}$. Note that $M$ spans more than $\alpha_{1} n$ edges, so $C_{R}$ must be connected (not a union of several red components). By Lemma \ref{lemma4.2}, at most $\gamma \varepsilon \alpha_{1} n$ of the edges in $M$ are missing in $G_{1}$. That means $C_{R}$ spans a matching on at least $\alpha_{1} n$ edges in $G_{1}$. By the construction of $G_{1}$, the component $C_{R}$ spans a matching with at least $\alpha_{1} n$ edges which shows that $G$ contains a red $\alpha_{1} n$-connected matching.\q

  In the same way, we can obtain:

  \begin{theo}\label{theo4.2}
    Let $0< \alpha_{1}< \alpha_{2}< \alpha_{3}\leq \frac{\alpha_{1}}{2}+ \alpha_{2}, \beta= (2\lceil \frac{\alpha_{1}+ 2\alpha_{3}+ 1}{\alpha_{1}} \rceil)^6,\gamma= \frac{12}{\alpha_{1}} \lceil \frac{\alpha_{1}+ 2\alpha_{3}+ 1}{\alpha_{1}} \rceil$. For every $0< \varepsilon <\frac{1}{\beta+ (\alpha_{1}+ 2\alpha_{3}) \gamma}$, there is $n_{0} = n_{0}(\varepsilon)$ such that the following holds. For $n> n_{0}$, let $G$ be a bipartite graph with partition $\{V_{1}, V_{2}\}$ and $|V_{1}|=|V_{2}|=N$, where $N\geq (\alpha_{1}+ 2\alpha_{3}+(\beta+ (\alpha_{1}+ 2\alpha_{3})\gamma )\varepsilon )n$. Suppose that every vertex in $V_{1}$ has at most $\varepsilon n$ non-neighbours in $V_{2}$ and vice versa. Then for every red-blue-green-edge-coloring of $G$, there is a red $\lfloor \alpha_{1}n \rfloor$-connected matching or a blue $\lfloor \alpha_{2}n \rfloor$-connected matching or a green $\lfloor \alpha_{3}n \rfloor$-connected matching.
  \end{theo}

  \begin{theo}\label{theo4.3}
    Let $0< \alpha_{1}< \alpha_{2}, \frac{\alpha_{1}}{2}+ \alpha_{2}< \alpha_{3}< \alpha_{1}+ \alpha_{2}, \beta= (2\lceil \frac{2\alpha_{1}+ 2\alpha_{2}+ 1}{\alpha_{1}} \rceil)^6,\gamma= \frac{12}{\alpha_{1}} \lceil \frac{2\alpha_{1}+ 2\alpha_{2}+ 1}{\alpha_{1}} \rceil$. For every $0< \varepsilon <\frac{1}{\beta+ (2\alpha_{1}+ 2\alpha_{2}) \gamma}$, there is $n_{0} = n_{0}(\varepsilon)$ such that the following holds. For $n> n_{0}$, let $G$ be a bipartite graph with partition $\{V_{1}, V_{2}\}$ and $|V_{1}|=|V_{2}|=N$, where $N\geq (2\alpha_{1}+ 2\alpha_{2}+(\beta+ (2\alpha_{1}+ 2\alpha_{2})\gamma )\varepsilon )n$. Suppose that every vertex in $V_{1}$ has at most $\varepsilon n$ non-neighbours in $V_{2}$ and vice versa. Then for every red-blue-green-edge-coloring of $G$, there is a red $\lfloor \alpha_{1}n \rfloor$-connected matching or a blue $\lfloor \alpha_{2}n \rfloor$-connected matching or a green $\lfloor \alpha_{3}n \rfloor$-connected matching.
  \end{theo}

  \subsection{Proof of Theorem \ref{theo2}}

  In this section, we will use Regularity Lemma and Theorems \ref{theo4.1}, \ref{theo4.2}, \ref{theo4.3} to complete the proof of Theorem \ref{theo2}. We only give the proof for the case $\alpha_{1}< \alpha_{2}, \alpha_{3}\geq \alpha_{1}+\alpha_{2}$, other two cases can be verified in the same way. 

  Let us recall some basic definitions related to the Regularity Lemma.

  \begin{defi}\label{def5.1}
    Let $A, B$ be disjoint subsets of vertices in a graph $G$. Denote the number of edges in $G$ with one endpoint in $A$ and another in $B$ by $e_{G}(A, B)$ and denote the edge density by $d_{G}(A,B) = \frac{e_{G}(A, B)}{|A||B|}$. Given $\varepsilon >0$, we say that the pair $(A, B)$ is $\varepsilon$-regular (with respect to the graph $G$) if for every $A'\subseteq A$ and $B'\subseteq B$ satisfying $|A'|\geq \varepsilon |A|$ and $|B'|\geq \varepsilon |B|$, we have $$|d_{G}(A',B') - d_{G}(A,B)|\leq \varepsilon.$$
  \end{defi}

  \begin{defi}\label{def5.2}
    A partition $\mathcal{P} = \{P_{0}, P_{1},...,P_{k}\}$ of the vertex set $V$ is said to be $(\varepsilon, k)$-equitable if $|P_{0}|\leq \varepsilon |V|$ and $|P_{1}| =...= |P_{k}|$. And an $(\varepsilon, k)$-equitable partition $\mathcal{P}$ is $(\varepsilon, k)$-regular if all but at most $\varepsilon \tbinom {k}{2}$ pairs $(P_{i}, P_{j})$ with $1\leq i< j\leq k$ are $\varepsilon$-regular.
  \end{defi}

  Szemer\'edi's regularity lemma states that for any $\varepsilon$ and $k_{0}$ there are $K_{0} = K_{0}(\varepsilon,k_{0})$ such that any graph admits an $(\varepsilon, k)$-regular partition with $k_{0}\leq k\leq K_{0}$. We will apply the following multicolored version of Regularity Lemma for bipartite graphs.

  \begin{lemma}[\cite{BLS}] \label{lemma6.1}
    For any $\varepsilon >0$ and $k_{0}$ there exist $K_{0} = K_{0}(\varepsilon ,k_{0})$, such that the following holds. Let $G$ be a 3-colored bipartite graph, with partition $\{V_{1}, V_{2}\}$, where $|V_{1}| = |V_{2}| = n$. Then there exists an $(\varepsilon, 2k)$-equitable partition $\mathcal{P} = \{V_{0}, U_{1}, U_{2},...,U_{k}, W_{1}, W_{2},..., W_{k}\}$ of $V(G)$ such that the following properties hold:

    (a) every $U_{i}$, for $i\geq 1$, is contained in $V_{1}$ and every $W_{j}$, for $j\geq 1$, is contained in $V_{2}$;

    (b) $|V_{0}\cap V_{1}|=|V_{0}\cap V_{2}|$;

    (c) $k_{0}\leq k\leq K_{0}$;

    (d) for every $i\in [k]$, for all but at most $\varepsilon k$ values of $j\in [k]$, $(U_{i},W_{j})$ is $\varepsilon$-regular with respect to each of colours of $G$.
  \end{lemma}

  \begin{defi}\label{def5.3}
    Given an edge-colored graph $G$ and a partition $\mathcal{P} = \{V_{0}, U_{1}, U_{2},...,U_{k},\\
    W_{1}, W_{2},..., W_{k}\}$, the $(\varepsilon,d)$-reduced graph $\Gamma$ is the graph whose vertices are $U_{1}, U_{2},...,U_{k}, \\
    W_{1}, W_{2},..., W_{k}$ and $U_{i}W_{j}$ is an edge if and only if $(U_{i}, W_{j})$ is $\varepsilon$-regular with respect to each colour of $G$ and its density in $G$ is at least $d$. We colour each edge $U_{i}W_{j}$ with majority color in $G[U_{i},W_{j}]$.
  \end{defi}

  The following lemma is used to lift a connected matching found in the reduced graph to a cycle in the original graph. It was proved by Figaj and {\L}uczak in \cite{LF}.

  \begin{lemma}\label{lemma6.2}
    Given $\varepsilon, d, k$ such that $0 <20\varepsilon <d <1$ there is an $n_{0}$ such that the following holds. Let $\mathcal{P}$ be an $(\varepsilon, k)$-equitable partition of a graph $G$ on $n\geq n_{0}$ vertices, and let $\Gamma$ be the corresponding $(\varepsilon, d)$-reduced graph. Suppose that $\Gamma$ contains a monochromatic $m$-connected matching. Then $G$ contains an even cycle of the same colour and of length $l$ for every even $l\leq 2(1-9\varepsilon d^{-1})m|U_{1}|$.
  \end{lemma}

  Now, we prove the following result.

  \begin{theo}
    Let $\alpha_{1}< \alpha_{2}, \alpha_{3}\geq \alpha_{1}+\alpha_{2}$, then $br(C_{2\lfloor \alpha_{1}n\rfloor}, C_{2\lfloor \alpha_{2}n\rfloor}, C_{2\lfloor \alpha_{3}n\rfloor}) \leq (\alpha_{1}+ \alpha_{2}+ \alpha_{3}+ o(1))n$ for $n$ sufficiently large.
  \end{theo}

  \begin{proof}
    Let $\mu >0$ and $N = (\alpha_{1}+\alpha_{2}+\alpha_{3}+\mu)n$. Suppose $n$ is sufficiently large and $\varepsilon'$ is sufficiently small. Apply the Regularity Lemma (Lemma \ref{lemma6.1}) to graph $G$ with parameter $\varepsilon'$, and let $\mathcal{P}$ be a partition satisfying the conditions of the lemma. Consider the corresponding $(\varepsilon' ,1)$-reduced graph $\Gamma$. Note that by $(a)$ and $(b)$ in Lemma \ref{lemma6.1}, $\Gamma$ is a balanced bipartite graph. Denote the number of vertices in each side by $k$ so that $\mathcal{P} = \{V_{0}, U_{1}, U_{2},...,U_{k}, W_{1}, W_{2},..., W_{k}\}$. Furthermore, every pair $(U_{i},W_{j})$ has density 1 in the original graph. Hence, $U_{i}W_{j}$ is an edge in $\Gamma$ if $(U_{i},W_{j})$ is $\varepsilon'$-regular with respect to each color. From $(d)$, $\Gamma$ has minimum degree at least $(1-2\varepsilon')k$.

    Let $n'=\frac{k}{\alpha_{1}+ \alpha_{2}+ \alpha_{3}+\zeta \varepsilon'}\geq \frac{k}{\alpha_{1}+ \alpha_{2}+ \alpha_{3}+1}$ ($\zeta =2(\alpha_{1} +\alpha_{2} +\alpha_{3}+1)(\beta+(\alpha_{1}+\alpha_{2}+ \alpha_{3}) \gamma)$, $\varepsilon'$ is sufficiently small and $\beta, \gamma$ are the same as in section 5.1). Since every vertex in one side of $\Gamma$ has at most $2\varepsilon' k\leq 2(\alpha_{1}+ \alpha_{2}+ \alpha_{3}+ 1) \varepsilon' n'$ non-neighbors. Apply Theorem \ref{theo4.1} by taking  $\varepsilon = 2(\alpha_{1}+ \alpha_{2}+ \alpha_{3}+ 1)\varepsilon' ,n= n' ,N =k$. Then $\Gamma$ contains a red $\alpha_{1}n'$-connected matching or a blue $\alpha_{2}n'$-connected matching or a green $\alpha_{3}n'$-connected matching. W.L.O.G, assume that $\Gamma$ contains a red $\alpha_{1}n'$-connected matching. Applying Lemma \ref{lemma6.2}, $G$ contains a red even cycle of length $l$ for any $l\leq 2(1-9\varepsilon')\alpha_{1} n' |U_{1}|$. Note that:
    \begin{equation*}
      \begin{aligned}
      2(1-9\varepsilon')\alpha_{1} n' |U_{1}|
      & = 2(1-9\varepsilon') \alpha_{1} \cdot \frac{k}{\alpha_{1}+ \alpha_{2}+ \alpha_{3}+\zeta \varepsilon'} \cdot |U_{1}| \\
      & \geq  2(1-9\varepsilon')(1-\varepsilon') \cdot \frac{\alpha_{1} N}{\alpha_{1}+ \alpha_{2}+ \alpha_{3}+\zeta \varepsilon'} \\
      & = 2(1-9\varepsilon')(1-\varepsilon') \cdot \frac{\alpha_{1}(\alpha_{1}+\alpha_{2}+ \alpha_{3}+\mu)n}{\alpha_{1}+ \alpha_{2}+ \alpha_{3}+\zeta \varepsilon'}\\
      & \geq 2\alpha_{1}n,
    \end{aligned}
    \end{equation*}
    where the first inequality follows as $k|U_{1}| =N-\frac{|V_{0}|}{2}\geq (1-\varepsilon')N$, and for the last since $\varepsilon'$ is sufficiently small compared to $\mu$. Hence, there is a red cycle of length at least $2\lfloor \alpha_{1}n\rfloor$.\q
  \end{proof}

  Applying the same procedure by  applying Theorem \ref{theo4.2} and Theorem \ref{theo4.3}, we can obtain:

  \begin{theo}\label{theo6.1}
    \[br(C_{2\lfloor \alpha_{1}n\rfloor}, C_{2\lfloor \alpha_{2}n\rfloor}, C_{2\lfloor \alpha_{3}n\rfloor})\leq \left \{
    \begin{array}{lll}
      (\alpha_{1}+ 2\alpha_{3}+ o(1))n &\mbox {if $\alpha_{1}< \alpha_{2}< \alpha_{3}\leq \frac{\alpha_{1}}{2}+\alpha_{2},$}\\
      (2\alpha_{1}+ 2\alpha_{2}+ o(1))n &\mbox {if $\alpha_{1}< \alpha_{2}, \frac{\alpha_{1}}{2}+\alpha_{2}< \alpha_{3}< \alpha_{1}+\alpha_{2},$}\\
      (\alpha_{1}+ \alpha_{2}+ \alpha_{3}+ o(1))n &\mbox {if $\alpha_{1}< \alpha_{2}, \alpha_{3}\geq \alpha_{1}+\alpha_{2}.$}
    \end{array}
    \right.
    \]
  \end{theo}

  It is easy to see that $br(C_{2\lfloor \alpha_{1}n\rfloor}, C_{2\lfloor \alpha_{2}n\rfloor}, C_{2\lfloor \alpha_{3}n\rfloor})\geq r(\lfloor \alpha_{1}n\rfloor, \lfloor \alpha_{2}n\rfloor, \lfloor \alpha_{3}n\rfloor)$. By Theorem \ref{theo3}, we have:

  \begin{theo}\label{theo6.2}
    \[br(C_{2\lfloor \alpha_{1}n\rfloor}, C_{2\lfloor \alpha_{2}n\rfloor}, C_{2\lfloor \alpha_{3}n\rfloor})\geq \left \{
    \begin{array}{lll}
      (\alpha_{1}+ 2\alpha_{3}+ o(1))n &\mbox {if $\alpha_{1}< \alpha_{2}< \alpha_{3}\leq \frac{\alpha_{1}}{2}+\alpha_{2},$}\\
      (2\alpha_{1}+ 2\alpha_{2}+ o(1))n &\mbox {if $\alpha_{1}< \alpha_{2}, \frac{\alpha_{1}}{2}+\alpha_{2}< \alpha_{3}< \alpha_{1}+\alpha_{2},$}\\
      (\alpha_{1}+ \alpha_{2}+ \alpha_{3}+ o(1))n &\mbox {if $\alpha_{1}< \alpha_{2}, \alpha_{3}\geq \alpha_{1}+\alpha_{2}.$}
    \end{array}
    \right.
    \]
  \end{theo}

  Combining Theorem \ref{theo6.1} and Theorem \ref{theo6.2}, we complete the proof of Theorem \ref{theo2}.

\end{document}